\documentclass[11 pt]{amsart}

\usepackage{lineno,hyperref}
\usepackage{amsmath,amssymb}
\usepackage{lineno}
\usepackage{centernot}
\theoremstyle{plain}
\newtheorem{thm}{Theorem}[section]
\newtheorem{theorem}[thm]{Theorem}
\newtheorem{lemma}[thm]{Lemma}
\newtheorem{corollary}[thm]{Corollary}
\newtheorem{proposition}[thm]{Proposition}

\theoremstyle{definition}
\newtheorem{definition}[thm]{Definition}
\newtheorem{example}[thm]{Example}

\newtheorem{remark}[thm]{Remark}

\numberwithin{equation}{section}

\newcommand{\interior}[1]{{\kern0pt#1}^{\mathrm{o}}}

\begin{document}


\title{Zero Insertive Nil Clean Rings}

\author[Sanjiv Subba]{Sanjiv Subba $^\dagger$}

\address{$^\dagger$School of Applied Sciences\\ UPES, Dehradun\\  Bidholi 248007\\ India}
\email{sanjivsubba59@gmail.com}
\author[Tikaram Subedi]{Tikaram Subedi  {$^{\dagger *}$}}
\address{$^{\dagger * }$Department of Mathematics\\  National Institute Of Technology  Meghalaya\\ Shillong 793003\\ India}
\email{tikaram.subedi@nitm.ac.in}

.


\begin{abstract}
 This paper investigates key properties of ZINC rings and their relationships with semicommutative and weakly semicommutative rings. We call an element $x$ of a ring $R$ zero insertive if $x=arb$ for some $a,b,r\in R$ such that $ab=0$ and $ZI(R)$ denotes the set of all zero insertive elements of  $R$. We establish that a ring $R$ is semicommutative if and only if $ZI(R) \subseteq E(R)$, and weakly semicommutative if and only if $ZI(R) \subseteq N(R)$, where $E(R)$ and $N(R)$ denote respectively the sets of idempotent elements and nilpotent elements. For ZINC rings with no nontrivial idempotents, $ZI(R) \subseteq N(R)$.

We prove that a finite direct product of ZINC rings is ZINC if and only if each component ring is ZINC, while an infinite direct product may fail to be ZINC. For $n \geq 2$, if $M_n(R)$ is ZINC, then $R$ is weakly clean, however, the converse is not true (e.g., $\mathbb{Z}$). Additionally, $M_n(K)$ is ZINC for a division ring $K$ if and only if $K \cong \mathbb{F}_2$. We, also, present a ZINC ring whose polynomial and power series extensions are not ZINC.

\end{abstract}

\maketitle
\maketitle
	\section{Introduction}
	
 \quad In \cite{Diesl}, Diesl introduced the concept of nil clean rings, which has since sparked the extensive study among mathematicians (see \cite{2x2nc}, \cite{ncmr}, \cite{Kosan}). An element of a ring is said to be  nil clean if it can be expressed as the sum of an idempotent and a nilpotent element. A ring $R$ is said to be nil clean if all its elements are nil clean.

This paper is concerned with the study of nil clean decomposition of certain specific elements of a ring $R$. We call an element $a$ of a ring $R$ zero insertive  if it can be expressed as $xry$, where $x,y,r\in R$ with $xy=0$. Zero insertive elements play a crucial role in generalizations of semicommutative rings, for instance,  in semicommutative rings  (\cite{wscr}) all zero insertive elements are zero and in weakly  semicommutative rings (\cite{wscr}) all zero insertive elements are nilpotent. In this paper, we study rings in which all zero insertive elements exhibit nil clean decomposition and call these rings as zero insertive nil clean (ZINC) rings.  Among other results, we prove that for $n\geq 2$ and  a division ring $K$,  $M_n(K)$ is ZINC if and only if $K\cong \mathbb{F}_2$; and  for a ring $R$, $M_n(R)$ is a ZINC ring, only if $R$ is weakly clean. 

\quad In this paper, $R$ represents an associative ring with unity, and all modules  are unital. We adopt the following notations: $U(R)$ for the set of all units of $R$, $Z(R)$ for the set of all central elements of $R$, $E(R)$ for the set of all idempotent elements of $R$, $J(R)$ for the Jacobson radical of $R$, $N(R)$ for the set of all nilpotent elements of $R$, $T_n(R)$ for the ring of upper triangular matrices of order $n \times n$ over $R$, and $M_n(R)$ for the ring of all $n \times n$ matrices over $R$. Moreover, we use the notation $E_{ij}$ for the matrix unit in $M_n(R)$ whose $(i, j)^{th}$ entry is $1$ and zero elsewhere.
	
\section{Main Results}

At the outset,  we begin by introducing  \textit{zero insertive} elements. Subsequently, we exploit these elements to investigate a new class of rings called \textit{zero insertive nil clean (ZINC)} rings.

\begin{definition}
We call an element $x\in R$  \textit{zero insertive}, if $x=arb$ for some $a,b,r\in R$ with $ab=0$. We denote $ZI(R)=\{x\in R$: $x$~{is~a~zero~ insertive~element~of}~$R\}$.
  \end{definition}

\begin{proposition}\label{semZI(R)}
	A ring $R$ is semicommutative if and only if $ ZI(R)\subseteq E(R)$.
\end{proposition}
\begin{proof}
	$(\Rightarrow)$. It is evident that in a semicommutative ring $R$, $ZI(R)=0$.\\ 
	$(\Leftarrow)$. Conversely, suppose that $ZI(R)\subseteq E(R)$. Let $x,y\in R$ be such that $xy=0$. For any $r\in R$,  $yxryx\in ZI(R)$. So by hypothesis and  $yxryx$ being a nilpotent,  $yxryx\in E(R)\cap N(R)$. Therefore,  $yxryx=0$. Since $xry\in ZI(R)\subseteq E(R)$, $xry=(xry)^3=xr(yxryx)ry=0$. Hence, $R$ is semicommutative.	
\end{proof}

Recall that $a$ ring $R$ is weakly semicommutative if for any $x,y\in R$, $xy=0$ implies $xRy\subseteq N(R)$. This definition can be rewritten in terms of $ZI(R)$ as follows:
\begin{definition}\label{wsemZI(R)}
	A ring $R$ is weakly semicommutative if and only if $ZI(R)\subseteq N(R)$.
\end{definition}

Proposition \ref{semZI(R)} and Definition \ref{wsemZI(R)} provide insights that motivate us to explore a generalization of semicommutative rings by incorporating the nil clean property. In this context, we introduce a new class of rings as follows:

\begin{definition}
	We call a ring $R$ \textit{zero insertive nil clean (ZINC)}  if $ZI(R)\subseteq E(R)+N(R)$.
\end{definition}

It is straightforward to note that weakly semicommutative rings are  ZINC. However, there exists a ZINC ring which is not weakly semicommutative. For instance, $M_2(\mathbb{Z}_2)$ is ZINC by Theorem \ref{dsnc}, but it is not weakly semicommutative as $E_{11}E_{12}E_{21}=E_{11}\notin N(M_2(\mathbb{Z}_2))$.

\begin{proposition}\label{wsn}
	Let $R$ be a  ZINC ring. If $R$ has no non trivial idempotent, then $ZI(R)\subseteq N(R)$.
\end{proposition}
\begin{proof}
	Let $x\in ZI(R)$. Then, $x=arb$ for some $a,b,r\in R$ and $ab=0$. Since $R$ is ZINC,  either  $x$ is nilpotent or $x=1+h$ for some $h\in N(R)$. If $x=1+h$, then $x\in U(R)$. So, $xv=arbv=1$ for some $v\in R$. Note that $rbva\in E(R)$, that is, $rbva=0$ or $1$. If $rbva=0$, then $a=0$, which implies $1+h=0$, a contradiction. If $rbva=1$, then $a$ is a unit. So, $b=0$, that is, $1+h=0$, a contradiction.
	
\end{proof}

\begin{proposition}
	Suppose $R$ is a ring with no non trivial idempotent. Then, $ZI(T_n(R))\subseteq E(T_n(R))+N(T_n(R))$ if and only if $ZI(T_n(R))\subseteq N(T_n(R))$.
\end{proposition}
\begin{proof}
	Suppose $ZI(T_n(R))\subseteq E(T_n(R))+N(T_n(R))$.	\\ Let $P=\begin{pmatrix}
		p_{11} &  p_{12}&  \dots &  p_{1n}\\
		0 &  p_{22} & \dots &  p_{2n}\\
		\vdots & \vdots& \ddots & \vdots \\
		0 & 0 & \dots &  p_{nn}
	\end{pmatrix}\vspace{0.2cm}\in ZI(T_n(R))$. This implies that each $p_{ii}\in ZI(R)$, $1\leq i\leq n$.  As $ZI(T_n(R))\subseteq E(T_n(R))+N(T_n(R))$, there exists $F=(f_{ij})\in E(T_n(R))$ such that $P-F\in N(T_n(R))$. By hypothesis, $f_{ii}=0$ or $1$, $1\leq i\leq n$. By the proof of Proposition \ref{wsn}, we get that $p_{ii}\in N(R)$, that is, $P\in N(T_n(R))$, that is, $ZI(T_n(R))\subseteq N(T_n(R))$. Whereas the converse is self-evident.                                                                                                    
\end{proof}

According to (\cite{sjc}), $a\in R$ is said to be strongly J-clean  if it can be expressed as the sum of an idempotent element and an element in its Jacobson radical that commute. A ring $R$ is said to be J-clean if for any $x\in R$, $x=e+j$ for some $e\in E(R), j\in J(R)$.

\begin{theorem}
	Let $R$ be a J-clean ring. If $J(R)$ is nil, then  $ZI(R)\subseteq N(R)$.
\end{theorem}

\begin{proof}
	Let $a\in ZI(R)$. So,  $a=xry$  for some $x,y,r\in R$ with $xy=0$. By hypothesis, there exists $e\in E(R)$ such that $a-e\in J(R)$. If $e=0$, we are done. Assume, if possible, that $e\neq 0$. Since $yx\in N(R)$, $1-yx\in U(R)$. As $R$ is J-clean, $1-yx-e_1\in J(R)$ for some $e_1\in E(R)$. Therefore, $1-(1-yx)^{-1}e_1\in J(R)$. This yields that $(1-yx)^{-1}e_1\in U(R)$ and hence $e_1\in U(R)$, that is, $e_1=1$. This implies that $yx\in J(R)$. Hence, $a^2\in J(R)$. Observe that
	$(a-e)^2=a^2-ae-e(a-e)$.  So, $ae\in  J(R)$. Note that $a-e=j$ for some $j\in J(R)$. Hence $e=ae-je\in J(R)$, that is, $e=0$, a contradiction. Thus, $a\in J(R)$, that is, $ZI(R)\subseteq N(R)$.
\end{proof}

\begin{theorem}\label{nil}
	Let $R$ be a ring and $I$ be a nil ideal of $R$. If $R/I$ is  ZINC, then $R$ is  ZINC.
\end{theorem}

\begin{proof}
	Suppose that $R/I$ is ZINC. Let $x\in ZI(R)$. Clearly, $\bar{x}\in ZI(R/I)$.  Since $R/I$ is  ZINC,  $\bar{x}=\bar{e}+\bar{h}$ for some $\bar{e}\in E(R/I), \bar{h}\in N(R/I)$. Since idempotents lifts modulo any nil ideal, $\bar{e}=\bar{f}$ for some $f\in E(R)$. Then $x-f$ is nilpotent modulo $I$. Thus, $(x-f)^m\in I\subseteq N(R)$ for some positive integer $m$. So, $x-f\in N(R)$. Therefore, $R$ is  ZINC.
\end{proof}

\begin{proposition}\label{prod}
	Let $\{R_{\alpha} \}$ be a finite collection of rings. Then the direct product $R=\Pi R_{\alpha}$  is  ZINC if and only if each $R_{\alpha}$ is  ZINC.
\end{proposition}

\begin{proof}
	The proof is trivial.	
\end{proof}

However, it is possible to  construct an infinite direct product of ZINC rings that is not  ZINC, as illustrated in the following Example \ref{infinite pdt}. The following remark will be referred in Example \ref{infinite pdt} and other subsequent results.

\begin{remark}\label{ZINC rem}
	For any $w\in R$, observe that $wE_{11}=\begin{pmatrix}
		w & 0 & \dots & 0\\
		0 & 0 & \dots & 0 \\
		\vdots & \vdots & \ddots & \vdots \\
		0 & 0 & \dots & 0
	\end{pmatrix}\in ZI(M_n(R))$ as $wE_{11}=E_{1n}(wE_{n1})E_{11}$ for all $n\geq 2$.
	
\end{remark} 

\begin{example}\label{infinite pdt}
	It is easy to observe that $\mathbb{Z}_{2^m}$ is nil clean for all positive integer $m$. By \cite[Corollary 7]{ncmr}, $M_n(\mathbb{Z}_{2^{m}})$ is nil clean for any positive intergers $n$ and $m$. Now, take $R=M_2(\mathbb{Z}_{2})\times M_3(\mathbb{Z}_{2^{2}})\times M_4(\mathbb{Z}_{2^{3}})\times \cdots=\Pi_{i=1}^{\infty} M_{i+1}(\mathbb{Z}_{2^{i}})$. Let $P=(2E_{11}, 2E_{11},\dots)$. By Remark \ref{ZINC rem}, $P\in ZI(R)$. Note that $J(\Pi_{i=1}^{\infty} M_{i+_1}(\mathbb{Z}_{2^{i}}))=\Pi_{i=1}^{\infty} M_{i+_1}(J(\mathbb{Z}_{2^{i}}))$. Let $K_i=2E_{11}\in M_{i+_1}(\mathbb{Z}_{2^{i}})$. Then $K_i\in J( M_{i+_1}(\mathbb{Z}_{2^{i}}))$ as $2\in J(\mathbb{Z}_{2^i})$. Suppose $K_i=F+B$, where $F$ is an idempotent and $B$ is a nilpotent element. Let $V=I+B\in U(M_{i+1}(\mathbb{Z}_{2^i}))$. So, $V=(I-F)+K_i$. This implies that $V=V(I-F)V^{-1}+VKV^{-1}$. So, we obtain $V(I-K_iV^{-1})=V(I-F)V^{-1}\in E(M_{i+1}(\mathbb{Z}_{2^i}))\cap U(M_{i+1}(\mathbb{Z}_{2^i}))$. Therefore, $F=0$. Note that $K_i$ is nilpotent of index $i$. Hence, $P\notin E(R)+N(R)$  as $P=(K_1, K_2,\dots)$. Thus, $R$ is not a  ZINC ring.   
\end{example}

The following corollary is a  consequence of Theorem \ref{nil} and Proposition \ref{prod}.

\begin{corollary} 
	\begin{enumerate}
		\item  $T_n(R)$ is  ZINC if and only if $R$ is  ZINC.
		\item Let $M$ be an $(R,R)$-bimodule. The trivial extension $R\propto M$ is ZINC if and only if $R$ is ZINC.
		\item $R[x]/\langle x^n \rangle$ is ZINC if and only if $R$ is ZINC.

	\end{enumerate}
\end{corollary}
\begin{proof}
	For  $(1)$: Suppose $R$ is  ZINC. Observe that $I=\begin{pmatrix}
		0 & R & \dots & R\\
		0 & 0 & \dots & R \\
		\vdots & \vdots & \ddots & \vdots \\
		0 & 0 & \dots & 0
	\end{pmatrix}$ is a nil ideal of $ T_n(R)$.  Note that $T_n(R)/I\cong R\times \dots \times R$.
	By Theorem \ref{nil} and Proposition \ref{prod}, $T_n(R)$ is ZINC. Conversely, suppose $T_n(R)$ is ZINC.  If $x\in ZI(R)$, then $A=\begin{pmatrix}
		x & \mathbf{0}\\
		\mathbf{0} & \mathbf{0}
	\end{pmatrix}\in ZI(T_n(R))$. Since $T_n(R)$ is ZINC, $x\in E(R)+N(R)$.
	The proofs of $(2)$ and $(3)$ follows analogously as $R\propto M\cong \left\{\left(\begin{array}{rr}
		t & s\\
		0 & t
	\end{array}
	\right) : t\in R, s\in M \right\}$ and  $R[x]/<x^n>\cong  
	\left \lbrace
	\left(\begin{array}{lccccr}
		r_1 & r_{2} & r_3 & \dots & r_{n-1} & r_n\\
		0 & r_1 & r_2 & \dots &  r_{n-2} & r_{n-1}\\
		\vdots & \vdots  & \vdots & \vdots & \vdots \\
		0 & 0 & 0 & \dots & r_1 & r_2 \\
		0 & 0 & 0 & \dots & 0 & r_1
	\end{array}
	\right ):r_i\in R\right \rbrace$ .
\end{proof}

\begin{lemma}\label{nztc}
	If $R$ is ZINC, then $ZI(R)\subseteq E(R)+U(R)$.
\end{lemma}
\begin{proof}
	Let $x\in ZI(R)$. Clearly, $-x\in ZI(R)$. As $R$ is ZINC, $-x=e+h$ for some $e\in E(R)$  and $h\in N(R)$. Therefore, $x=(1-e)+(-h-1)\in E(R)+U(R)$.	
\end{proof}

We skip the proof of the following trivial Lemma. 
\begin{lemma}\label{invertible}
	Let $R$ be a ring. Then,   $U=\begin{pmatrix}
		1 & u_{12} & \dots & u_{1n}\\
		0 & 1 & \dots & u_{2n} \\
		\vdots & \vdots & \ddots & \vdots \\
		0 & 0 & \dots & 1
	\end{pmatrix}, V=\begin{pmatrix}
	1 & 0 & \dots & 0\\
	v_{21} & 1 & \dots & 0 \\
	\vdots & \vdots & \ddots & \vdots \\
	v_{n1} & v_{n2} & \dots & 1
\end{pmatrix}\in M_n(R)$ are invertible.
\end{lemma}

\begin{lemma}\label{ZINCmatrix}
	Suppose $X=\begin{pmatrix}
		1 & u_{12} & \dots & u_{1n}\\
		0 & 1 & \dots & u_{2n} \\
		\vdots & \vdots & \ddots & \vdots \\
		0 & 0 & \dots & 1
	\end{pmatrix}$, $Y=\begin{pmatrix}
		1 & 0 & \dots & 0\\
		-x_2 & 1 & \dots & 0 \\
		\vdots & \vdots & \ddots & \vdots \\
		-x_n & 0 & \dots & 1
	\end{pmatrix}$ and $U=(u_{ij})\in M_n(R)$. If $XYU=(b_{ij})$, then $b_{1j}=(1-\sum\limits_{i=2}^nu_{1i}x_i)u_{1j}+ \sum\limits_{k=2}^nu_{1k}u_{kj}$, when $j\geq 1$ and $	b_{ij}=(-x_i-\sum\limits_{k=i+1}^nu_{ik}x_k)u_{1j}+u_{ij}+\sum\limits_{k=i+1}^nu_{ik}u_{kj}$, when $i> 1$, $j\geq 1$.
\end{lemma}

A ring $R$ is said to be weakly clean (\cite{wclean}) if for any $x\in R$ there exist $e\in E(R), u\in U(R)$ such that $x-e-u\in (1-e)Rx$.

\begin{theorem}\label{weaklyc}
	For any $n\geq 2$, if $M_n(R)$ is a ZINC ring, then $R$ is weakly clean ring.
\end{theorem}
\begin{proof}
	Let $w\in R$ and $ W=wE_{11}\in M_n(R)$. By Remark \ref{ZINC rem}, $W\in ZI(M_n(R))$. Since  $M_n(R)$ is ZINC,   $W=F+V$ for some $F\in E(M_n(R)$ and $V\in U(M_n(R))$. This implies that $V^{-1}W=V^{-1}F+I$, where $I$ is the identity matrix. So, we obtain \vspace{0.2cm} $V^{-1}W=(V^{-1}FV)V^{-1}=GV^{-1}+I$, where  $G=V^{-1}FV$ which is also idempotent. Hence, $(I-G)V^{-1}W=I-G$. As $W=wE_{11}$, $I-G=\begin{pmatrix}
		g & 0 & \dots & 0\\
		x_2 & 0 & \dots & 0 \\
		\vdots & \vdots & \ddots & \vdots \\
		x_n & 0 & \dots & 0
	\end{pmatrix}\vspace{0.2cm}$ for some $x_i,g\in R$, $2\leq i\leq n$.  Since $I-G$ is idempotent,   $g\in E(R), x_i\in Rg$, and  $i>1$. So,
	
	\begin{equation}\label{ZINCG}
		G=\begin{pmatrix}
			f & 0 & \dots & 0 & 0\\
			-x_2 & 1 & \dots & 0 & 0 \\
			\vdots & \vdots & \ddots & \vdots & \vdots \\
			-x_{n-1} & 0 & \dots & 1 & 0\\
			-x_n & 0 & \dots & 0 &  1
		\end{pmatrix}
	\end{equation}
	where $f=1-g$.
	Let $V^{-1}=(u_{ij})$. Since $V^{-1}W=GV^{-1}+I$ \vspace{0.2cm}, \\ $\begin{pmatrix}
		u_{11}w & 0 & \dots & 0\\
		u_{21}w & 0 & \dots & 0 \\
		\vdots & \vdots & \ddots & \vdots \\
		u_{n1}w & 0 & \dots & 0
	\end{pmatrix}=\begin{pmatrix}
		fu_{11}+1 & fu_{12} & \dots & fu_{1n}\\
		-x_2u_{11}+u_{21} & -x_2u_{12}+u_{22}+1 & \dots & -x_2u_{1n}+u_{2n} \\
		\vdots & \vdots & \ddots & \vdots \\
		-x_nu_{11}+u_{n1} & -x_nu_{12}+u_{n2} & \dots & -x_nu_{1n}+u_{nn}+1
	\end{pmatrix}\vspace{0.2cm}$ (see equation \ref{ZINCG}). So, we obtain the following equalities \ref{ZINC1}, \ref{ZINC2}, \ref{ZINC3}, \ref{ZINC4} and \ref{ZINC5} :
	
	\begin{equation}\label{ZINC1}
		fu_{11}+1=u_{11}w
	\end{equation} For $j>1$, we have
	\begin{equation}\label{ZINC2}
		fu_{1j}=0=(1-g)u_{1j}	
	\end{equation} For $1<j\leq n$, we have 
	\begin{equation}\label{ZINC3}
		-x_ju_{1j}+u_{jj}=-1	
	\end{equation} For $1<i\leq n$, we have
	\begin{equation}\label{ZINC4}
		-x_iu_{11}+u_{i1}=u_{i1}w	
	\end{equation}
	For all $i\neq j$,  $1<i,j\leq n$ 
	
	\begin{equation}\label{ZINC5}
		-x_iu_{1j}+u_{ij}=0	
	\end{equation}
	
	If $B=\begin{pmatrix}
		1 & u_{12} & \dots & u_{1n}\\
		0 & 1 & \dots & u_{2n} \\
		\vdots & \vdots & \ddots & \vdots \\
		0 & 0 & \dots & 1
	\end{pmatrix}\begin{pmatrix}
		1 & 0 & \dots & 0\\
		-x_2 & 1 & \dots & 0 \\
		\vdots & \vdots & \ddots & \vdots \\
		-x_n & 0 & \dots & 1
	\end{pmatrix}V^{-1}$, then   $B$  is invertible (see Lemma \ref{invertible}). Let $B=(b_{ij})$. We claim that $b_{11}\in U(R)$. Now, for determining $b_{ij}$, we use the Lemma \ref{ZINCmatrix} freely. So, we get

	$\begin{aligned}
		b_{11} & =(1-\sum\limits_{i=2}^nu_{1i}x_i)u_{11}+ \sum\limits_{k=2}^nu_{1k}u_{k1}\\ 
		& =u_{11}+u_{12}(-x_2u_{11}+u_{21})+u_{13}(-x_3u_{11}+u_{31})+\dots +u_{1n}(-x_nu_{11}+u_{n1})
	\end{aligned}$
	
	Therefore, by equation \ref{ZINC4}
	\begin{equation}\label{ZINC6}
		b_{11}=u_{11}+\sum\limits_{i=2}^nu_{1i}u_{i1}w	
	\end{equation} For $j> 1$, we have
	
	$\begin{aligned}
		b_{1j} & =(1-\sum\limits_{i=2}^nu_{1i}x_i)u_{1j}+ \sum\limits_{k=2}^nu_{1k}u_{kj}\\
		&=u_{1j}+u_{12}(-x_2u_{1j}+u_{2j})+\dots +u_{1j}(-x_ju_{1j}+u_{jj})+\dots +u_{1n}(-x_nu_{1j}+u_{nj})	
	\end{aligned}$
	Therefore, by equation  \ref{ZINC3}, equation \ref{ZINC5}, for $j>1$
	
	\begin{equation}\label{ZINC7}
		b_{1j}=0
	\end{equation}

	For $1<j<i$, 
	
	$\begin{aligned}
		b_{ij}& =(-x_i-\sum\limits_{k=i+1}^nu_{ik}x_k)u_{1j}+u_{ij}+\sum\limits_{k=i+1}^nu_{ik}u_{kj}\\
		&=-x_iu_{1j}+u_{ij}+\sum\limits_{k=i+1}^nu_{ik}(-x_ku_{1j}+u_{kj}).	
	\end{aligned}$ \\
	
	So, by equation \ref{ZINC5},  for $1<j<i$
	\begin{equation}
		b_{ij}=0
	\end{equation}
	
	For $i\neq 1$, we have
	
	$\begin{aligned}
		b_{ii} &=(-x_i-\sum\limits_{k=i+1}^nu_{ik}x_k)u_{1i}+u_{ii}+\sum\limits_{k=i+1}^nu_{ik}u_{ki}\\
		& =-x_iu_{1i}+u_{ii}+\sum\limits_{k=i+1}^nu_{ik}(-x_ku_{1i}+u_{ki}).
	\end{aligned}$

	Therefore, by equation \ref{ZINC3} and equation \ref{ZINC5}, for $i\neq 1$
	\begin{equation}
		b_{ii}=-1
	\end{equation}
	
	For $1<i<j$,
	
	$\begin{aligned}
		b_{ij} &=(-x_i-\sum\limits_{k=i+1}^nu_{ik}x_k)u_{1j}+u_{ij}+\sum\limits_{k=i+1}^nu_{ik}u_{kj}\\
		&=-x_iu_{1j}+u_{ij}+\sum\limits_{{k=i+1},k\neq j}^nu_{ik}(-x_ku_{1j}+u_{kj})+u_{ij}(-x_ju_{ij}+u_{jj}).
	\end{aligned}$

	Hence, by equation \ref{ZINC3} and equation \ref{ZINC5}, $1<i<j$, we have 
	\begin{equation}
		b_{ij}=-u_{ij}
	\end{equation}

	Therefore, $B=\begin{pmatrix}
		u_{11}+\sum\limits_{i=2}^nu_{1i}u_{i1}w & 0 & 0 & \dots & 0 & 0\\
		b_{21} & -1 & -u_{23} & \dots & -u_{2,n-1} & -u_{2n} \\
		\vdots & \vdots & \vdots & \vdots & \vdots & \vdots \\
		b_{n-1,1} & 0 & 0 & \dots & -1 & -u_{n-1,n}\\
		b_{n1} & 0 & 0 & \dots & 0 & -1
	\end{pmatrix}\vspace{0.2cm}$. 
	
	Take $c_{i1}=b_{i1}+\sum\limits_{j=i+1}^nb_{ij}(b_{j1}+\sum\limits_{k=j+1}b_{jk}b_{k1})$ and  $C=\begin{pmatrix}
		1 & 0 & 0 & \dots & 0 & 0\\
		c_{21} & 1 & 0 & \dots & 0 & 0 \\
		\vdots & \vdots & \vdots & \vdots & \vdots & \vdots \\
		c_{n-1,1} & 0 & 0 & \dots & 1 & 0\\
		c_{n1} & 0 & 0 & \dots & 0 & 1
	\end{pmatrix}\vspace{0.3cm}$. Note that $B$ and $C$ are units (see Lemma \ref{ZINCmatrix}) and so is \\ $BC=\begin{pmatrix}
		b_{11} & 0 & 0 & \dots & 0  & 0\\
		0 & -1 & -u_{23} & \dots  & -u_{2,n-1} & -u_{2n} \\
		\vdots & \vdots & \vdots & \vdots & \vdots & \vdots \\
		0 & 0 & 0 & \dots & -1 & -u_{n-1,n}\\
		0 & 0 & 0 & \dots & 0 & -1
	\end{pmatrix}\vspace{0.2cm}\in U(M_n(R))$. Hence, $b_{11}=u_{11}+\sum\limits_{i=2}^nu_{1i}u_{i1}w\in U(R)$. Substituting $u_{11}=b_{11}-(\sum\limits_{i=2}^nu_{1i}u_{i1})w$ in equation \ref{ZINC1}, we get $f(b_{11}-(\sum\limits_{i=2}^nu_{1i}u_{i1})w)+1=(b_{11}-(\sum\limits_{i=2}^nu_{1i}u_{i1})w)w$. By equation \ref{ZINC2}, we obtain that $fb_{11}+1=b_{11}w-(\sum\limits_{i=2}^nu_{1i}u_{i1})w^2$, that is, $w-b_{11}^{-1}fb_{11}-b_{11}^{-1}=b_{11}^{-1}(\sum\limits_{i=2}^nu_{1i}u_{i1})w^2=b_{11}^{-1}(\sum\limits_{i=2}^ngu_{1i}u_{i1})w^2$ (by equation \ref{ZINC2}, $gu_{1i}=u_{1i}$ for all $i>1$). So, $w-e-v=(1-e)rw^2$, where $e=b_{11}^{-1}fb_{11}$, $v=b_{11}^{-1}$ and $r=(\sum\limits_{i=2}^nvu_{1i}u_{i1})$. Therefore, $R$ is weakly clean. 
\end{proof}

 \begin{corollary}\label{pm}
	If $R$ is a domain such that $M_n(R)$ is ZINC, then for each $w\in R$ there exists $v\in U(R)$ such that $w-v=rw^2$ or $w-v=1$ for some $r\in R$.
\end{corollary}

\begin{corollary}
	For any integer $n>1$, $M_n(\mathbb{Z})$ is not ZINC.
\end{corollary}
\begin{proof}
	For any $u\in U(\mathbb{Z})=\{\pm 1\}$, we have 
	$3-u=2$ or $4$. By Corollary \ref{pm}, $M_n(\mathbb{Z})$ is not ZINC. 
\end{proof}
\begin{remark}
	Observe that all commutative rings are ZINC but need not be weakly clean, for instance, $\mathbb{Z}$. Therefore, Theorem \ref{weaklyc} is not true for $n=1$, in general.
\end{remark} 

\begin{lemma} \label{edia} \cite[Corollary 5]{dia}
	Let $A\in E(M_n(R))$. If $A$ is equivalent to a diagonal matrix, then $A$ is similar to a diagonal matrix. 
\end{lemma}

When we write a matrix $X\in M_n(R)$  in block form: $X=  \begin{pmatrix}
	(X)_{11} & (X)_{12}\\
	(X)_{21} & (X)_{22}
\end{pmatrix}$, it means $(X)_{11}, (X)_{12},  (X)_{21}, (X)_{22}$ have  orders $1\times 1, 1\times (n-1), (n-1)\times 1$ and $(n-1)\times (n-1)$, respectively.

\begin{lemma}\label{Zblock}
	Let $W$ and $Z$ be matrices over $R$ of order $1\times n-1$ and $n-1\times 1$, $n>1$. If $X\in M_{n}(R)$ and $(X)_{11}=a, (X)_{12}=bW, (X)_{21}=Zc$ and $(X)_{22}=ZdW$, $a,b,c,d\in R$, then  $(X^l)_{11}=a_l, (X^l)_{12}=b_lW, (X^l)_{21}=Zc_l$ and $(X^l)_{22}=Zd_lW$ for all positive integer $l$, and $a_l, b_l, c_l, d_l\in R$.	
\end{lemma} 
\begin{proof}
	Observe that $X=\begin{pmatrix}
		a &  \hspace{0.2cm} bW\\
		Zc & \hspace{0.2cm} ZdW
	\end{pmatrix}$. By induction, it is easy to see that
	$X^l=\begin{pmatrix}
		a_l &  \hspace{0.2cm} b_lW\\
		Zc_l & \hspace{0.2cm} Zd_lW
	\end{pmatrix}$, where $a_l, b_l, c_l, d_l\in R$ for all positive integer $l$. So, $(X^l)_{11}=a_l,~ (X^l)_{12}=b_lW,~ (X^l)_{21}=Zc_l,~ (X^l)_{22}=Zd_lW$.
\end{proof}

\begin{theorem}\label{dsnc}
	If $K$ is a division ring and $n$ an integer $\geq 2$, then $M_n(K)$ is ZINC if and only if $K\cong \mathbb{F}_2$. 
\end{theorem} 

\begin{proof}
	Let $M_n(K)$ be a ZINC ring. Assume, if possible, that $K$ is not isomorphic to $\mathbb{F}_2$. Then, there exists $x \in K\setminus \{0,1\}$. Then,   $P=xE_{11}\in ZI(M_n(\mathbb{K}))$ (by Remark \ref{ZINC rem}). Since $M_n(K)$ is ZINC, $P=F+N$, for some idempotent matrix $F$ and nilpotent matrix $N$ in $M_n(K)$. It is already  well established  that every matrix over a division ring is equivalent to a diagonal matrix. Therefore, by Lemma \ref{edia}, $F$ is similar to a diagonal matrix whose diagonal entries are $0$ or $1$. So, there exists a unit $A=(a_{ij})\in M_n(K)$, such that 
	$A^{-1}FA=\begin{pmatrix}
		I_l & \mathbf{0}\\
		\mathbf{0} & \mathbf{0}
	\end{pmatrix}$, where $I_l$ is the identity matrix in $M_l(K)$, $l\leq n$.  Hence, 
	\begin{equation}\label{N'}
		A^{-1}PA=\begin{pmatrix}
			I_l & \mathbf{0}\\
			\mathbf{0} & \mathbf{0}
		\end{pmatrix} + A^{-1}NA=\begin{pmatrix}
			I_l & \mathbf{0}\\
			\mathbf{0} & \mathbf{0}
		\end{pmatrix} + N^{'}
	\end{equation} where $N'=A^{-1}NA$. Observe that $A^{-1}PA$ is not nilpotent. So, $l\geq 1$. Note that $l\neq n$, otherwise $P=I_n+N$ is invertible, a contradiction as $P$ is not a unit. Therefore, $1\leq l< n$. Observe that $I_n+N'\in U(M_n(K))$ and so $A(I_n+N')\in U\left(M_n(K)\right)$.  With the help of the equation \eqref{N'}, we obtain that
$A(I_n+N')=A\begin{pmatrix}
I_l & \mathbf{0}\\
\mathbf{0} & \mathbf{0}
\end{pmatrix}\vspace{0.2cm}+A\begin{pmatrix}
\mathbf{0} & \mathbf{0}\\
\mathbf{0} & I_{n-l}
\end{pmatrix} +AN'=PA+A\begin{pmatrix}
\mathbf{0} & \mathbf{0}\\
\mathbf{0}  & I_{n-l}
\end{pmatrix}= \begin{pmatrix}
xa_{11} & \dots & xa_{1n}\\
0 & \dots & 0 \\
\vdots & \vdots & \vdots \\
0 &\dots & 0
\end{pmatrix}+ \begin{pmatrix}
0 & \dots& 0 & {a_{1,l+1}} & \dots & a_{1n} \\
0 & \dots & 0 &  a_{2,l+1} & \dots & a_{2n} \\
\vdots & \vdots & \vdots & \vdots  & \vdots & \vdots  \\
0 & \dots & 0 & a_{n,l+1} & \dots & a_{nn}   
\end{pmatrix}\vspace{0.2cm}$\\
$=\begin{pmatrix}
xa_{11} & \dots& xa_{1l} & (1+x){a_{1,l+1}} & \dots & (1+x)a_{1n} \\
0 & \dots & 0 &  a_{2,l+1} & \dots & a_{2n} \\
\vdots & \vdots & \vdots & \vdots  & \vdots & \vdots  \\
0 & \dots & 0 & a_{n,l+1} & \dots & a_{nn}   
\end{pmatrix}\vspace{0.2cm}$. This implies that $l=1$ (otherwise $A(I_n+N')\notin U\left(M_n(K)\right)$) , $a_{11}\neq 0$ and $\begin{pmatrix}
a_{22} & \dots& a_{2n}\\
\vdots & \ddots& \vdots \\
a_{n2}& \dots & a_{nn}
\end{pmatrix}\vspace{0.2cm}=A_1$ (say) is a unit in $M_{n-1}(K)$. Therefore, equation \eqref{N'} becomes
$PA=\begin{pmatrix}
x & \mathbf{0}\\
\mathbf{0} & \mathbf{0}
\end{pmatrix}A=A\begin{pmatrix}
1 & \mathbf{0}\\
\mathbf{0} & \mathbf{0}
\end{pmatrix}+AN'$ and so we can write
\begin{equation}
\begin{pmatrix}
	a_{11}^{-1} & \mathbf{0} \\
	\mathbf{0} & A_1^{-1}
\end{pmatrix} \begin{pmatrix}
	x & \mathbf{0} \\
	\mathbf{0} & \mathbf{0}
\end{pmatrix}\begin{pmatrix}
	a_{11} & \mathbf{0} \\
	\mathbf{0} & A_1
\end{pmatrix}\begin{pmatrix}
	a_{11}^{-1} & \mathbf{0} \\
	\mathbf{0} & A_1^{-1}
\end{pmatrix}A=\begin{pmatrix}
	a_{11}^{-1} & \mathbf{0} \\
	\mathbf{0} & A_1^{-1}
\end{pmatrix}A\begin{pmatrix}
	1 & \mathbf{0} \\
	\mathbf{0} & \mathbf{0}
\end{pmatrix}+\begin{pmatrix}
	a_{11}^{-1} & \mathbf{0} \\
	\mathbf{0} & A_1^{-1}
\end{pmatrix}AN'
\end{equation}
So, we have
\begin{equation}\label{N''}
\begin{pmatrix}
	y & \mathbf{0}\\
	\mathbf{0} & \mathbf{0}
\end{pmatrix}U=U \begin{pmatrix}
	1 & \mathbf{0}\\
	\mathbf{0} & \mathbf{0}
\end{pmatrix} +UN'   
\end{equation} where $y=a_{11}^{-1}xa_{11}$ and $U= \begin{pmatrix}
a_{11}^{-1} & \mathbf{0}\\
\mathbf{0} & {A}_{1}^{-1}
\end{pmatrix}A= \begin{pmatrix}
1 & {W}\\
{Z} & \hspace{0.2cm} {{I}}_{n-1}
\end{pmatrix}$ (say).  Let $U^{-1}= \begin{pmatrix}
b & {W}'\\
{Z}' & {T}
\end{pmatrix}$. We have $UU^{-1}={I}_n=U^{-1}U$. So, we obtain the following equations:

\begin{equation}\label{DM1}
b+{WZ}'=b+{W}'{Z}=1	
\end{equation}

\begin{equation}\label{DM2}
{W}'+{WT}=\textbf{0}=b{W}+{W}	'
\end{equation}

\begin{equation}\label{DM3}
{ZW}'+{T}={I}_{n-1}={Z}'{W}+{T}	
\end{equation}

\begin{equation}\label{DM4}
{Z}b+{Z}'=\textbf{0}={Z}'+{TZ}
\end{equation}

From equations \eqref{DM1}, \eqref{DM2},  \eqref{DM3} and  \eqref{DM4}, we get  $U^{-1}= \begin{pmatrix}
b & -b{W}\\
-{Z}b & \hspace{0.4cm} {I}_{n-1}+{Z}b{W}
\end{pmatrix}$.

From equations \eqref{DM1} and \eqref{DM2}, we get 
\begin{equation}
1=b+{WZ}'=(1-{WZ})b	
\end{equation}	
So, $b\neq 0$ and ${WZ}=1-b^{-1}$.  If ${WZ}=0$, then $b=1$, and $U^{-1}= \begin{pmatrix}
1 & -{W}\\
-{Z} & \hspace{0.4cm} {I}_{n-1}+{ZW}
\end{pmatrix}$. By equation \eqref{N''}, we have\\

$UN'U^{-1}= \begin{pmatrix}
y & \textbf{0}\\
\textbf{0} & \textbf{0}
\end{pmatrix}-U \begin{pmatrix}
1 & \textbf{0}\\
\textbf{0} & \textbf{0}
\end{pmatrix}U^{-1}= \begin{pmatrix}
y-1 & {W}\\
-{Z} & {ZW}
\end{pmatrix}\vspace{0.2cm}$. For all integers $l\geq 1$, we get that $(UN'U^{-1})^{l+1}= \begin{pmatrix}
(y-1)^{l+1} & (y-1)^l{W}\\
-{Z}(y-1)^l & \hspace{0.2cm}-{Z}(y-1)^{l-1}{W}
\end{pmatrix}\vspace{0.2cm}\neq \textbf{0}$ as $y-1\neq 0$. This is a contradiction as $UN'U^{-1}$ is nilpotent.  Therefore, we can assume that ${WZ}\neq 0$. Now, by equation \eqref{N''}, we have 
\begin{equation}
\begin{pmatrix}
	1 & -{W}\\
	\mathbf{0} & \hspace{0.2cm}{I}_{n-1}
\end{pmatrix} \begin{pmatrix}
	y & \mathbf{0}\\
	\mathbf{0} & \mathbf{0}
\end{pmatrix}  \begin{pmatrix}
	1 & {W}\\
	\mathbf{0} & \hspace{0.2cm} {I}_{n-1}
\end{pmatrix}\begin{pmatrix}
	1 & -{W}\\
	\mathbf{0} & \hspace{0.2cm}{I}_{n-1}
\end{pmatrix}U= \begin{pmatrix}
	1 & -{W}\\
	\mathbf{0} & \hspace{0.2cm} {I}_{n-1}
\end{pmatrix}U \begin{pmatrix}
	1 & \mathbf{0}\\
	\mathbf{0} & \mathbf{0}
\end{pmatrix}+ \begin{pmatrix}
	1 & -{W}\\
	\mathbf{0} & \hspace{0.2cm}{I}_{n-1}
\end{pmatrix}UN'
\end{equation}
So,
\begin{equation}\label{N'''}
\begin{pmatrix}
	y & y{W}\\
	\mathbf{0} & \textbf{0}
\end{pmatrix}S=S  \begin{pmatrix}
	1 & \mathbf{0}\\
	\mathbf{0} & \mathbf{0}
\end{pmatrix}+SN'  
\end{equation}
where $S=  \begin{pmatrix}
1 & -{W}\\
\mathbf{0} & \hspace{0.2cm} {I}_{n-1}
\end{pmatrix}U=  \begin{pmatrix}
b^{-1} & \mathbf{0}\\
{Z} & \hspace{0.2cm}{I}_{n-1}
\end{pmatrix}$ and $S^{-1}=  \begin{pmatrix}
b & \mathbf{0}\\
-{Z}b & \hspace{0.2cm}{I}_{n-1}
\end{pmatrix}$. Let $Q=SN'S^{-1}$.  From equation \eqref{N'''}, we get 
\begin{equation}\label{Na}
Q=  \begin{pmatrix}
	y & \hspace{0.2cm}y{W}\\
	\mathbf{0} & \mathbf{0}
\end{pmatrix}-S  \begin{pmatrix}
	1 & \mathbf{0}\\
	\mathbf{0} & \mathbf{0}
\end{pmatrix}S^{-1}= \begin{pmatrix}
	y-1 & \hspace{0.2cm} y{W}\\
	-{Z}b & \textbf{0}
\end{pmatrix}
\end{equation}
From equation \eqref{Na} and Lemma \eqref{Zblock}, we obtain
\begin{equation}\label{N;'}
(Q^l)_{12}=b_l{W}, \hspace{0.2 cm}(Q^l)_{21}={Z}c_l, \hspace{0.2 cm} (Q^l)_{22}={Z}a_l{W} 
\end{equation}  for some $ a_l, b_l ,c_l\in K$.
Since $Q$ is nilpotent and $(Q)_{21}=-{Z}b\neq \textbf{0}$, there a positive integer $i$ such that $(Q^{i+1})_{21}=\textbf{0}$ but $(Q^i)_{21}\neq \textbf{0}$. Therefore, by equations \eqref{Na}, and  \eqref{N;'}
\begin{equation*}	
Q^{i+1}=\begin{pmatrix}
	(Q^i)_{11} & \hspace{0.3 cm} (Q^i)_{12}\\
	(Q^i)_{21} & \hspace{0.3 cm} (Q^i)_{22}
\end{pmatrix}\begin{pmatrix}
	(Q)_{11} & \hspace{0.3 cm}  y{W}\\
	(Q)_{21} & \hspace{0.3 cm} \textbf{0}
\end{pmatrix}=\begin{pmatrix}
	(Q^{i+1})_{11} & \hspace{0.3 cm} (Q^{i+1})_{12}\\
	\textbf{0} & \hspace{0.3 cm} {Z}c_iy{W}
\end{pmatrix}
\end{equation*}
where $c_iy\neq 0$. Now, we see that if $t\in K$, ${Z}t{W}=\textbf{0}$, then $t=0$. So, if $t\neq 0$, $({Z}t{W})^2={Z}(t{WZ}t){W}\neq \textbf{0}$ (since $tWZt\neq 0$) and in the same way $({Z}t{W})^l\neq \textbf{0}$ for all $l\geq 0$. So, ${Z}c_iy{W}\in M_{n-1}(K)$ is not nilpotent, that is, $Q^{l+1}$ is not nilpotent, a contradiction. Hence, $x=0$ or $1$, that is, $K\cong \mathbb{F}_2$. On the other hand, the converse follows from \cite[Theorem 3]{ncmr}.
\end{proof}

\begin{remark}
	Any division ring is weakly clean. If $R$ is a division ring with $|R|>2$  then $M_n(R)$ is not ZINC (by Theorem \ref{dsnc}), where  $|R|$ denotes  the cardinality of $R$. Therefore, the converse of Theorem \ref{weaklyc} is not true.
	
\end{remark}

It is evident that any quotient ring of a nil-clean ring is nil-clean. However, there exists a ZINC ring whose quotient ring is not  ZINC ring.
\begin{example}\label{frac}
	Let $R$ be the localization of the ring $\mathbb{Z}$ at  $3\mathbb{Z}$ and $S$ the quaternions over $R$, that is, a free $R$-module with basis $1,~i,~j,~k$ and satisfy $i^2=j^2=k^2=-1,~ij=k=-ji$. Clearly, $S$ is a non-commutative domain. So, $S$ is a ZINC ring.  By \cite[Example 3]{yh},  $J(S)=3S$, and $S/3S\cong M_2(\mathbb{Z}_3)$ which in not ZINC by Theorem \ref{dsnc}. Therefore, $R/J(R)$ is not ZINC. 
\end{example}

A \textit{Morita context} (\cite{morita}) is a $4$-tuple $\begin{pmatrix}
	R_1 & M\\
	P & R_2
\end{pmatrix}$, where $R_1$, $R_2$ are rings, $M$ is $(R_1,R_2)$-bimodule and $P$ is $(R_2,R_1)$-bimodule, and there exists a context product $M\times P\rightarrow R_1$ and $P\times M\rightarrow R_2$ written multiplicatively as $(m,p)\mapsto mp$ and $(p,m)\mapsto pm$. Clearly, $\begin{pmatrix}
	R_1 & M\\
	P & R_2
\end{pmatrix}$ is an associative ring with the usual matrix operations.\\
A Mortia context $\begin{pmatrix}
	R_1 & M\\
	P & R_2
\end{pmatrix}$ is said to be trivial if the context products are trivial, that is, $MP=0$ and $PM=0$.

\begin{theorem}
	Let $R=\begin{pmatrix}
		A & M\\
		N & B
	\end{pmatrix}$ be a trivial Morita context. If $A$ and $B$ are ZINC, then $R$ is ZINC.
	
\end{theorem}
\begin{proof}
  Suppose $A$ and $B$ are  ZINC rings. Let $S\in ZI(R)$. Then, $S=PXQ$, where  $P=\begin{pmatrix}
			a_0 & m_0 \\
			n_0 & b_0
		\end{pmatrix} $, $Q=\begin{pmatrix}
			a_1 & m_1 \\
			n_1 & b_1
		\end{pmatrix}, X=\begin{pmatrix}
		x & m_2 \\
		n_2 & y
	\end{pmatrix} \in R$ with $PQ=0.$ Since $PQ=0$, $a_0a_1=0$ and $b_0b_1=0$. Observe that $a_0xa_1\in ZI(A)$ and $b_0yb_1\in ZI(B)$. Since $A$ and $B$ are ZINC rings, there exist $e_x\in E(A),~e_y\in E(B)$, $n_x\in N(A),~m_y\in N(B)$ such that $a_0xa_1=e_x+n_x$\vspace{0.2cm} and $b_0yb_1=e_y+n_y$.
		Now, $PXQ=\begin{pmatrix}
			e_x & 0 \\
			0 & e_y
		\end{pmatrix}\vspace{0.2cm} +\begin{pmatrix}
			n_x & m \\
			n & n_y
		\end{pmatrix}$ for some $m\in M$ and $n\in N$. Observe that $\begin{pmatrix}
			e_x & 0 \\
			0 & e_y
		\end{pmatrix}\in E(R)$, $\begin{pmatrix}
			n_x & m \\
			n & n_y
		\end{pmatrix}\in N(R)$. Therefore, $R$ is ZINC.
			
\end{proof}

\begin{proposition}
	The following are equivalent:
	\begin{enumerate}
		\item $R$ is ZINC. 
		\item $Re$ is ZINC for all central idempetent $e$.
		\item There exists a central idempotent $e$ of $R$ such that both $Re$ and $R(1-e)$ are ZINC.
	\end{enumerate}
\end{proposition}
\begin{proof}
	It is easy to observe that $(1)\Rightarrow (2)\Rightarrow (3)$.
	Now, for $(3)\Rightarrow (1)$, let $x\in ZI(R)$. Observe that $x=xe+x(1-e)$. As $e$ is central, $xe\in ZI(Re)$ and $x(1-e)\in ZI(R(1-e))$. By hypothesis, $x=(f+p)+(g+q)$, where $f\in E(Re), p\in N(Re), g\in E(R(1-e)), q\in N(R(1-e))$. Hence, $x=(f+g)+(p+q)\in E(R) +N(R)$. 
\end{proof}

\begin{proposition}\label{sj} 
	Let $R$ be a  ZINC ring.  If $ab=0$, then $aJ(R)b\subseteq N(R)$.
	
\end{proposition}
\begin{proof}
	Let $a,b\in R$ be such that $ab=0$. Since $R$ is ZINC, $aJ(R)b\subseteq ZI(R) \subseteq E(R)+N(R)$. So, for any $j\in J(R)$,  there exist $e\in E(R)$, $h\in N(R)$ such that $ajb=e+h$.  Then $(ajb-e)^{m}=0$ for some positive integer $m$. This implies that $j_1-e=0$ for some $j_1\in J(R)$.  So, $e\in J(R)$, that is, $e=0$.
\end{proof}

If $R=M_2(\mathbb{R})$, then $AJ(R)B=0$ for any $A,B\in R$. By Theorem \ref{dsnc}, the converse of Proposition \ref{sj} is not true.  

Now, in the following example we observe that there exists a ZINC ring whose polynomial extension and power series extension are not ZINC.
\begin{example}
	
	\begin{enumerate}
		
		\item Take $R=M_2(\mathbb{Z}_2)$, which is ZINC by Theorem \ref{dsnc}.  Observe that $f(x)=E_{12}$, $g(x)=E_{11}x\in R[x]$ and $f(x)g(x)=0$. So, $f(x)E_{21}g(x)=g(x)\in ZI(R[x])$. Suppose $g(x)=e(x)+n(x)$, where $e(x)=e_0+e_1x+\dots +e_kx^k\in E(R[x]), n(x)=n_0+n_1x+\dots +n_{k_1}x^{k_1}\in N(R[x])$. Clearly, $e_0=-n_0\in E(R)\cap N(R)$, that is, $e_0=0$. If $e(x)\neq 0$, then $e(x)=x^mh(x)$, where $m\geq 1$ and $h(x)=h_0+h_1x+h_2x^2+\dots +h_{k_2}x^{k_2} \in R[x]$ and $h_0\neq 0$. Comparing the coefficient of $x^m$ in $e(x)^2=e(x)$ gives $h_0=0$, a contradiction. Thus $g(x)\in N(R[X])$, a contradiction. Therefore, $R[x]$ is not ZINC.
		
		\item  Take $R=M_2(\mathbb{Z}_2)$ which is ZINC (by Theorem \ref{dsnc}).  It is well known that $J(R[[x]])=J(R)+\langle x \rangle$. Observe that $E_{11}E_{21}=0$ and $E_{11}(E_{12}x)E_{21}=E_{11}x\notin N(R[[x]])$. By Proposition \ref{sj}, $R[[x]]$ is not ZINC.  
	\end{enumerate}
	
\end{example}


\section*{Competing interests} 
The authors have no conflict of interest to declare that are
relevant to the content of this study.

\end{document}